\documentclass[10pt,a4paper]{article}
\usepackage[utf8]{inputenc}
\usepackage[english]{babel}
\usepackage{multicol,graphicx,color}
\usepackage{pslatex}
\usepackage{authblk}
\usepackage{amsthm}
\usepackage{amsmath}
\usepackage{amssymb}
\usepackage{latexsym}
\usepackage{lscape}
\usepackage{epsfig}
\usepackage{amsfonts}
\usepackage{mathrsfs,bbm}
\usepackage[
hmarginratio={1:1},     
vmarginratio={1:1},     
textwidth=15cm,        
textheight=21cm,
heightrounded,          
]{geometry}

\usepackage{tikz}

\makeatletter

\newtheorem{theorem}{Theorem}[section]

\newtheorem{definition}[theorem]{Definition}

\newtheorem{lemma}[theorem]{Lemma}
\newtheorem{proposition}[theorem]{Proposition}
\newtheorem{remark}[theorem]{Remark}

\theoremstyle{definition} \theoremstyle{remark}
\numberwithin{equation}{section}

\newcommand\R{{\mathbb R}}
\newcommand{\G}{\mathcal{G}}

\newcommand{\La}{\Lambda}
\newcommand{\la}{\lambda}
\newcommand{\Ga}{\Gamma}
\newcommand{\ga}{\gamma}

\begin{document}

	\title{Multiplicity result for a mass supercritical NLS with a partial confinement
	}

\author{Louis Jeanjean\footnote{louis.jeanjean@univ-fcomte.fr}}  \affil{{\it Universit\'e Marie et Louis Pasteur, CNRS, LmB (UMR 6623), F-25000, Besan\c{c}on, France}}

\author{Linjie Song\footnote{songlinjie18@mails.ucas.edu.cn}}  \affil{{\it Tsinghua University, Department of Mathematical Sciences, Beijing 100084, China}}
	
\date{}
\maketitle

\begin{abstract}
We consider an NLS equation in $\R^3$ with partial confinement and mass supercritical nonlinearity. 
In Bellazzini, Boussaid, Jeanjean and Visciglia (Comm. Math. Phys. 353, 2017, 229-251) for such a problem, a solution with a prescribed $L^2$ norm was obtained, as a local minimum of the energy, and the existence of a second solution, at a mountain pass level, was proposed as an open problem. We give here a positive, non-perturbative, answer to this problem. Our solution is obtained as a limit of a sequence of solutions of corresponding problems on bounded domains of $\R^3$. The symmetry of solutions on bounded domains is used centrally in the convergence process.
\end{abstract}

\medskip

{\small \noindent \text{Key Words:} NLS equations; Constrained functionals; $L^2$-supercritical.
	
\medskip

}

\medskip

{\small \noindent \text{Acknowledgements:} This work has been carried out in the framework of the Project NQG (ANR-23-CE40-0005-01), funded by the French National Research Agency (ANR). It was completed when the second author was visiting the Universit\'e Marie et Louis Pasteur, LmB (UMR 6623). Linjie Song is supported by "Shuimu Tsinghua Scholar Program" and by "National Funded Postdoctoral Researcher Program" (GZB20230368), and funded by "China Postdoctoral Science Foundation" (2024T170452). The first author thanks A. Farina for its help on proving Theorem \ref{thmasymonBR}.
	}

\medskip

{\small \noindent \text{Statements and Declarations:} The authors have no relevant financial or non-financial interests to disclose.}

\medskip

{\small \noindent \text{Data availability:} Data sharing is not applicable to this article as no datasets were generated or analysed during the current study.}

\section{Introduction} \label{secint}

In this paper we investigate the existence and multiplicity of critical points for the \emph{mass supercritical} NLS energy functional with a \emph{partial confinement} $E: H \to \R$ defined by
\begin{align*} 
	 E(u) = \frac12\int_{\mathbb{R}^3}(|\nabla u|^2 + (x_{1}^{2} + x_{2}^{2})|u|^2)dx - \frac1{p+1}\int_{\mathbb{R}^3}|u|^{p+1}dx, \quad \frac73 < p <5,
\end{align*}
constrained on the $L^2$ sphere
\begin{align*} 
	S_\mu := \bigl\{u \in H: \int_{\mathbb{R}^3}|u(x_1,x_2,x_3)|^2dx = \mu\}, \quad \mu > 0,
\end{align*}
where $x = (x_1,x_2,x_3) \in \R^3$ and
\begin{align*}
	H := \bigl\{u \in H^1(\R^3): \int_{\R^3}(x_1^2+x_2^2)|u|^2dx < \infty \bigr\}.
\end{align*}
The norm of $H$ is
\begin{align*}
	\|u\|_H^2 := \int_{\R^3}(|\nabla u|^2 + (x_1^2+x_2^2+1)|u|^2)dx.
\end{align*}
In the sequel, we denote $V(x_1,x_2,x_3) = x_1^2+x_2^2$. Critical points, also called bound states, solve the stationary nonlinear Schr\"odinger equation (NLS) with the partial confinement on $\R^3$
\begin{align} \label{eqequation}
	-\Delta u + V(x_1,x_2,x_3)u + \la u = |u|^{p-1}u
\end{align}
for some Lagrange multiplier $\la$. In turn, solutions to \eqref{eqequation} give standing waves of the time-dependent NLS with partial confinement on $\R^3$,
\begin{align} \label{timenls}
	\mathrm{i}\, \partial_t \psi
	= -\Delta_{x} \psi + V(x_1,x_2,x_3)\psi -|\psi|^{p-1} \psi,
\end{align}
via the ansatz $\psi(t,x) = e^{\mathrm{i} \lambda t}u(x)$. The mass constraint $\int_{\R^3}|u|^2dx = \mu$ is meaningful from a dynamics perspective as the mass (or charge), as well as the energy, is conserved by the NLS flow.

We point out that the physically relevant cubic nonlinearity $p = 3$ is covered by our assumptions. Cubic NLS, often referred as Gross-Pitaevskii equation (GPE), is very important in physics and one of the main physical motivations is the study of matter waves in the theory of Bose-Einstein condensates (BEC). Bose-Einstein condensate consists in a macroscopic ensemble of bosons that at very low temperature occupy the same quantum states. From the mathematical side, the properties of BEC at temperature much smaller than the critical temperature can be well described by the three-dimensional Gross-Pitaevskii equation (GPE) for the macroscopic wave function. GPE incorporates information about the trapping potential as well as the interaction among the particles. Our assumptions cover the case of Bose-Einstein condensate with attractive interactions and partial confinement. The latter includes the limit case of the so-called cigar-shaped model, see \cite{BC}. We refer to \cite{BBJV} for additional elements regarding the physical motivations and to \cite{ArCa,ChGuGuLiSh,HoJi, Oh} for recent mathematical contributions on the NLS with a partial confinement.

To our knowledge the existence of critical points for mass supercritical NLS energy functional with a partial confinement $E$ was first touched by \cite{BBJV}, in which a (positive) local minimizer was found when $0 < \mu < \mu_0$ where $\mu_0$ is a positive constant (see \cite[Theorem 1]{BBJV}). Further, assuming, in addition, that $\mu > 0$ is small enough \cite[Theorem 3]{BBJV} showed that such local minimizer is a ground state in the sense introduced in \cite{BJ}.
\begin{definition}[Ground state] \label{defgs}
	A critical point $u \in S_\mu$ of $E$ constrained on $S_\mu$ is called a (normalized) ground state if and only if
	\begin{align*}
		E(u) = \inf \bigl\{E(v): v \text{ is a critical point of } E \text{ constrained on } S_\mu \bigr\}.
	\end{align*}
\end{definition}
\noindent Recently \cite{SY} provided a new, direct, proof of the existence of a ground state of $E$ which holds for any $0 < \mu < \tilde\mu_0$. Here  $\tilde\mu_0$ is an explicit positive constant not obtained by a limit process, as in \cite[Theorem 3]{BBJV}. It was also proved that this is true for the mass critical case $p = 7/3$.

It was mentioned in \cite [Remark 1.10]{BBJV} that $E$ has a mountain pass geometry on $S_\mu$ for all $\mu$ for which there exists a local minimizer, and, that an interesting, but maybe difficult, problem would be to prove the existence of a second critical point for $E$ at the mountain pass level. 

In \cite{WW}, J. Wei and Y. Wu gave a first answer to this open problem. They proved that $E$ has indeed a second positive critical point, which is at a mountain pass level, when $\mu >0$ is sufficiently small. However, this result is obtained using a perturbative approach, so it is not possible to estimate how small $\mu >0$ must be. As far as we know, there is no non-perturbation result up to now, and to provide a non-perturbation result is the main motivation of this paper.

We introduce
\begin{align} \label{eqG}
	\G := \bigl\{u \in H: \int_{\R^3}|\nabla u|^2dx > \frac{3p-3}{2p+2}\int_{\R^3}|u|^{p+1}dx \bigr\}, \quad \G_\mu := \G \cap S_\mu,
\end{align}
with
\begin{align*}
	\partial \G_\mu = \bigl\{u \in S_\mu: \int_{\R^3}|\nabla u|^2dx = \frac{3p-3}{2p+2}\int_{\R^3}|u|^{p+1}dx \bigr\},
\end{align*}
and define
\begin{align}\label{min_boundary}
	 c_\mu := \inf_{u \in \G_\mu}E(u), \quad \mbox{and} \quad
	 \nu_\mu := \inf_{u \in \partial \G_\mu}E(u).
\end{align}


It is standard to prove, see for example \cite{BJ}, that any critical point of $E$ satisfies the following Pohozaev identity
\begin{align} \label{eqpohozaev}
	\int_{\R^3}|\nabla u|^2dx = \int_{\R^3}V|u|^2 dx + \frac{3p-3}{2p+2}\int_{\R^3}|u|^{p+1}dx.
\end{align}
The introduction of the subset $\G_\mu$ was  majorly inspired by \eqref{eqpohozaev}. It is readily seen that $E$ is bounded from below on $\G_\mu$ and thus $c_\mu > -\infty$ (in fact, $c_\mu > 0$ when $7/3 < p < 5$) is well defined. When $0 < \mu < \mu_1$ where $\mu_1$ is given by \eqref{eqmu1} later, we have $c_\mu < \nu_\mu$. Then any minimizer of $E$ constrained on $\G_\mu$ is in the interior of $\G_\mu$ and is a critical point of $E$ constrained on $S_\mu$. Further, by \eqref{eqpohozaev}, any critical point of $E$ constrained on $S_\mu$ is in $\G_\mu$. Hence, every minimizer of $E$ in $\G_\mu$ is a ground state. We will indeed show the existence of a ground state at the level $c_\mu$ when $0 < \mu < \mu_1$. 

After finding a ground state, we search for a critical point at a mountain pass level. When $0 < \mu < \mu_1$, we prove the existence of $0 \leq w_0 \in \G_\mu$, $0 \leq w_1 \in S_\mu \backslash \G_\mu$ such that $E(w_i) < \nu_\mu$, $i = 0,1$, and 
\begin{align*}
	m_\mu := \inf_{\ga \in \Ga}\sup_{t \in [0,1]}E(\ga(t)) \geq \nu_\mu > \max\bigl\{E(w_0),E(w_1)\bigr\},
\end{align*}
where
\begin{align} \label{eqGa}
	\Ga := \bigl\{\ga \in C([0,1],S_\mu): \ga(0) = w_0, \ga(1) = w_1\bigr\}.
\end{align}
This shows that $E$ indeed has a mountain pass geometry on $S_\mu$ when $0 < \mu < \mu_1$. We will establish the existence of a critical point at the level $m_\mu$.

We denote $C_{p+1} > 0$ the best constant of Gagliardo-Nirenberg inequality on $\R^3$, i.e., 
\begin{align} \label{eqgninequ}
	\int_{\R^3}|u|^{p+1}dx \leq C_{p+1}\left(\int_{\R^3}|\nabla u|^2dx \right)^{\frac{3p-3}{4}} \left(\int_{\R^3}|u|^2dx \right)^{\frac{5-p}{4}}
\end{align}
and let
\begin{align}\label{infspectrum}
\La_0 := \inf (\hbox{ \em spec }(-\Delta + V)) = \inf_{\int_{\R^3}|w|^2dx = 1}\int_{\R^3}\left(|\nabla w|^2 +(x_1^2+x_2^2)|w|^2 \right)dx.
\end{align}
Our main result is as follows.
\begin{theorem} \label{thmmulti}
	Let $7/3 < p <5$ and
	\begin{align} \label{eqmu1}
		\mu_{1} := \frac{(3p-7)^{\frac{3p-7}{2p-2}}(2p+2)^{\frac{2}{p-1}}}{(3p-3)^{\frac32}C_{p+1}^{\frac{2}{p-1}}\La_0^{\frac{2p-2}{3p-7}}}
	\end{align}
	where  $C_{p+1} >0$  and $\La_0$ are defined in \eqref{eqgninequ}  and \eqref{infspectrum} respectively. Then, for any $0 < \mu < \mu_1$, there is at least two positive critical points $u_1, u_2$ of $E$ constrained on $S_\mu$, $u_1$ is a ground state at the level $c_\mu$ and $u_2$ is at the mountain pass level $m_\mu$.
\end{theorem}


\begin{remark} \label{rmkomega1}
	As well known, see, for example, \cite[Lemma 3.1]{BaJe}, for any $\mu > 0$, there exists a unique $\omega_\mu > 0$ such that the equation
	\begin{align}\label{limitequation}
	-\Delta u + \omega_\mu u = u^p \quad \text{in } \R^3, \quad\quad u(x) \to 0 \quad \text{as } |x| \to \infty,
	\end{align}
	has a positive solution in $H^1(\R^3)$ with $\int_{\R^3}|u|^2dx = \mu$. Such positive solution is unique. We denote it $u_{\omega_{\mu}}$.
\end{remark}

Our second result provides properties of the critical points obtained in Theorem \ref{thmmulti}.
\begin{theorem} \label{thmasym}
	Let $u_1, u_2$ be the critical points of $E$ constrained on $S_\mu$ given by Theorem \ref{thmmulti}. Then $u_i(x_1,x_2.x_3)$ is radially symmetric and decreasing w.r.t. $(x_1,x_2)$ and w.r.t. $x_3$, $i = 1,2$. Moreover, there
	exists $\la_i = \la(u_i)$ such that
	\begin{align*}
		-\Delta u_i + Vu_i +\la_i u_i = |u_i|^{p-1}u_i \quad \text{in} \ \mathbb{R}^{3},
	\end{align*}
    with the estimates
    \begin{align}
    	& -\La_0 < \la_1 < -\La_0\left( 1 - C_{p+1}\left( \frac{3p-3}{3p-7}\La_0\right)^{\frac{3p-7}{4}}\mu^{\frac{p-1}{2}} \right) \quad \text{when } 0 < \mu < C_{p+1}^{-\frac2{p-1}}\left( \frac{3p-3}{3p-7}\La_0\right)^{-\frac{3p-7}{2p-2}}, \label{eqestiofla1} \\
    	& \la_2 = \left( \omega_1 + o(1)\right) \mu^{-\frac{2p-2}{3p-7}} \quad \text{as } \mu \to 0, \label{eqestiofla2}	
    \end{align} 
    where $\omega_1 > 0$ is given by Remark \ref{rmkomega1}. In addition, as $\mu \to 0$,
    \begin{align}
    	& \|u_1 - \Psi_0(x_1,x_2)\varphi_0(x_3)\|_{H} = o(\sqrt{\mu}), \label{eqestiofu1} \\
    	& \mu^\frac{2}{3p-7}u_2(\mu^{\frac{p-1}{3p-7}}x) \to u_{\omega_1} \quad \text{strongly in } H^1(\R^3), \label{eqestiofu2}
    \end{align}
    where $\Psi_0(x_1,x_2)$ is the unique normalized positive eigenvector of the quantum harmonic oscillator $-\Delta_{x_1,x_2} +x_1^2+x_2^2$, 
    $$
    \varphi_0(x_3)= \int_{\R^2} u(x_1,x_2,x_3) \Psi_0(x_1,x_2) dx_1dx_2,
    $$
    and $u_{\omega_1} \in H^1(\R^3)$ is the unique positive solution of \eqref{limitequation}.
\end{theorem}

As mentioned above, the existence of a ground state in Theorem \ref{thmmulti} is not new. However, our proof is quite different from the ones in \cite{BBJV} and in \cite{SY}. The existence of a critical point at the mountain pass level in Theorem \ref{thmmulti} is original and is the major motivation of our study. The main difficulty is the lack of compactness, due to the translation invariance w.r.t. $x_3$. Indeed if one assumes that the potential is fully harmonic, namely that $x_1^2 + x_2^2$ is replaced by $x_1^2 + x_2^2 + x_3^2$ then one benefits from the compactness of the inclusion of $H$ (with $x_1^2 + x_2^2 + x_3^2$) into $L^q(\R^3)$ where $q \in [2, 6)$ see for
example \cite[Lemma 3.1]{Zhang} and the proof of Theorem \ref{thmmulti} would be rather simple. To overcome this lack of compactness, we first consider the problem on $B_R := \{x \in \R^3: |x| < R\}$ with $R$ large enough. For $R > 0$ and $u \in H^1_0(B_R)$, we define
\begin{align*} 
	E_R(u) = \frac12\int_{B_R}(|\nabla u|^2 + (x_{1}^{2} + x_{2}^{2})|u|^2)dx - \frac1{p+1}\int_{B_R}|u|^{p+1}dx, \quad \frac73 < p <5,
\end{align*}
and
\begin{align*} 
	S_{\mu,R} := \bigl\{u \in H^1_0(B_R): \int_{B_R}|u(x_1,x_2,x_3)|^2dx = \mu\}, \quad \mu > 0.
\end{align*}
We further introduce
\begin{align*} 
	\G_R := \bigl\{u \in H^1_0(B_R): \int_{B_R}|\nabla u|^2dx > \frac{3p-3}{2p+2}\int_{B_R}|u|^{p+1}dx \bigr\}, \quad \G_{\mu,R} := \G_R \cap S_{\mu,R},
\end{align*}
with
\begin{align*}
	\partial \G_{\mu,R} = \bigl\{u \in S_{\mu,R}: \int_{B_R}|\nabla u|^2dx = \frac{3p-3}{2p+2}\int_{B_R}|u|^{p+1}dx \bigr\}.
\end{align*}
For $0 < \mu < \mu_1$ and sufficiently large $R$ (depending on $\mu$) $\G_{\mu,R}$ is nonempty and we define
\begin{align*}
	c_{\mu,R} := \inf_{u \in \G_{\mu,R}}E_R(u).
\end{align*}
Moreover, without loss of generality, we can assume that $w_0$ and $w_1$ used to define $\Ga$ in \eqref{eqGa} have compact support contained in $B_R$. Hence, setting $w_{0,R} \equiv w_0, w_{1,R} \equiv w_1$, we have that
\begin{align*}
	m_{\mu,R} := \inf_{\ga \in \Ga_R}\sup_{t \in [0,1]}E_R(\ga(t)) \geq \nu_{\mu} > \max\bigl\{E_R(w_{0,R}),E_R(w_{1,R})\bigr\},
\end{align*}
where
\begin{align*} 
	\Ga_R := \bigl\{\ga \in C([0,1],S_{\mu,R}): \ga(0) = w_{0,R}, \ga(1) = w_{1,R}\bigr\}.
\end{align*}
We will search for critical points of $E_R$ constrained on $S_{\mu,R}$ at levels $c_{\mu,R}$ and $m_{\mu,R}$. Any critical point of $E_R$ constrained on $S_{\mu,R}$ satisfies the following Pohozaev identity
\begin{align} \label{eqpohozaevonBR}
	\int_{B_R}|\nabla u|^2dx = \frac R2\int_{\partial B_R} \Bigl|\frac{\partial u}{\partial n}\Bigr|^2d\sigma + \int_{B_R}V|u|^2 dx + \frac{3p-3}{2p+2}\int_{B_R}|u|^{p+1}dx,
\end{align}
and thus is in $\G_{\mu,R}$. Therefore, critical points at the level $c_{\mu,R}$ are ground states of $E_R$.

\begin{theorem} \label{thmmultionBR}
	Let $7/3 < p <5$ and $0 < \mu < \mu_{1}$ where $\mu_1$ is given by Theorem \ref{thmmulti}.
	There exists $R_0 = R_0(\mu) > 0$ such that for all $R > R_0$, there is at least two positive critical points $u_{1,R}, u_{2,R}$ of $E_R$ constrained on $S_{\mu,R}$, $u_{1,R}$ is a ground state at the level $c_{\mu,R}$ and $u_{2,R}$ is at the mountain pass level $m_{\mu,R}$.
\end{theorem}

\begin{theorem} \label{thmasymonBR}
	Let $u_{1,R}, u_{2,R}$ be the critical points of $E_R$ constrained on $S_{\mu,R}$ given by Theorem \ref{thmmultionBR}. For $i = 1,2$, there
	exists $\la_{i,R} = \la(u_{i,R})$ such that
	\begin{align*}
		-\Delta u_{i,R} + Vu_{i,R} +\la_{i,R} u_{i,R} = |u_{i,R}|^{p-1}u_{i,R} \quad \text{in} \ B_R,
	\end{align*}
	with $\la_{i,R} > - \La_{R}$, where $\La_{R}$ is the first eigenvalue of $-\Delta + V$ on $B_R$ with Dirichlet boundary condition. In addition, $u_{i,R} \in C^2(\overline{B_R})$, $u_{i,R} >0$ on $B_R$ and $u_{i,R}(x_1,x_2.x_3)$ is radially symmetric and decreasing w.r.t. $(x_1,x_2)$ and w.r.t. $x_3$.
	
\end{theorem}

To prove Theorem \ref{thmmultionBR}, we notice that
any minimizing sequence in $\G_{\mu,R}$ at the level $c_{\mu,R}$ is bounded (see Lemma \ref{lempropofgmu}) and that any bounded sequence has a strongly convergent subsequence (see Lemma \ref{lemcompactonBR}). Thus we can find $u_{1,R}$ immediately when $\G_{\mu,R}$ is non-empty and the condition $c_{\mu,R} < \inf_{\partial \G_{\mu,R}}E$ holds. The key difficulty to search for $u_{2,R}$ is to establish the existence of a bounded Palais-Smale sequence at the level $m_{\mu,R}$. For this we shall introduce a family of approximating functionals and make use of the monotonicity trick (see Lemma \ref{lemappropro} and the proof of Theorem \ref{thmmultionBR}). As for Theorem \ref{thmasymonBR}, it is readily seen that $u_{1,R}$  and $u_{2,R}$ are positive and regular and that the associated Lagrange multipliers satisfy $\la_{1,R}, \la_{2,R} > - \La_{R}$. Furthermore, a result from  \cite{GNN} allows to deduce that our solutions are radially symmetric and decreasing w.r.t. the $(x_1,x_2)$ variables and $x_3$ variable respectively. We refer to Section \ref{secBR} for more details.

With Theorems \ref{thmmultionBR} and \ref{thmasymonBR} in hand, we turn to the proof of Theorem \ref{thmmulti}.  As we shall see, our proof of Theorem \ref{thmmulti} will ultimately rely on a Liouville type result (Proposition \ref{propnoposisolu}) saying that any non-negative continuous function which is radially symmetric and decreasing w.r.t. $(x_1,x_2)$ and w.r.t. $x_3$  and solution to
\begin{align*}
	-\Delta u + Vu + \lambda u = u^p \quad \text{in } \R^3
\end{align*}
must be the zero function when $\lambda = -\La_0$.  Such Liouville type results when $V \equiv 0$ and $\la = 0$ or when $V$ is harmonic potential and $-\la$ is the first eigenvalue are well known but we could not find a reference in the partial confinement case with $V = x_1^2 + x_2^2$. To prove Proposition \ref{propnoposisolu}, we use reductio ad absurdum. Elaborately choosing a family of test functions depending the parameter $R$ and sending $R$ to infinity, then we analyze the convergent rate to find a contradiction. 

The main lines of the proof of Theorem \ref{thmmulti} are the following. 
Theorem \ref{thmmultionBR} provides a sequence $R_n \to \infty$, together with sequences $\{u_{1,n}\}, \{u_{2,n}\}$ such  that $u_{i,n} \in S_{\mu,R_n}$ is a positive critical point of $E_{R_n}$ constrained on $S_{\mu,R_n}$, $i = 1,2$, and $E_{R_n} (u_{1,n}) = c_{\mu,R_n}$, $E_{R_n} (u_{2,n}) = m_{\mu,R_n}$. Such critical points have the symmetry properties mentioned in Theorem \ref{thmasymonBR}.  Being constituted of critical points, it immediately follows from the Pohozaev identity \eqref{eqpohozaevonBR} that $\{u_{1,n}\}, \{u_{2,n}\}$ are bounded in $H$.

In Lemma \ref{lempropofg}, Lemma \ref{lemmps} and Lemma \ref{lemlimitofcandm}, we prove that $c_{\mu,R} \to c_\mu < 1/2\La_0\mu$ and $m_{\mu,R} \to m_\mu > 1/2\La_0\mu$ as $R \to \infty$. This implies that $E_{R_n} (u_{1,n}) \to c_\mu < 1/2\La_0\mu$ and $E_{R_n} (u_{2,n}) \to m_\mu > 1/2\La_0\mu$ as $n \to \infty$.

A key result is then Theorem \ref{thmcompactness}. It states that if $\{u_n\}$ is a sequence such that $u_n \in S_{\mu,R_n}$ is a positive critical point,  having the symmetry properties stated in Theorem \ref{thmasymonBR}, of $E_{R_n}$ constrained on $S_{\mu,R_n}$, with $R_n \to \infty$ and $E_{R_n}(u_n) \to \beta \neq 1/2\La_0\mu$, then, up to a subsequence, $u_n \to u$ strongly in $H$. Note that a direct application of Theorem \ref{thmcompactness} then guarantees the strong convergence of $\{u_{1,n}\}$ and $\{u_{2,n}\}$ in $H$ and thus ends the proof of Theorem \ref{thmmulti}.

The assumption  $\beta \neq 1/2\La_0\mu$, in Theorem \ref{thmcompactness}, is used to show that the sequence $\{u_n\}$ is non vanishing. Then, the symmetry properties of $u_n$ implies the strong convergence of $\{u_n\}$  in $L^{p+1}(\R^3)$ (Lemma \ref{lemstrongconvergence}).  In view of this last convergence, the convergence in $L^2(\R^3)$ becomes a direct consequence of the Liouville result mentioned above.  We refer to Section \ref{subseccompactness} for more details.

Regarding now Theorem \ref{thmasym}, the estimates \eqref{eqestiofla1} and \eqref{eqestiofu1} can be obtained along the line of \cite[Subsection 4.3]{BBJV} and \cite[Subsection 4.4]{BBJV}. Further, we use some scaling arguments to establish \eqref{eqestiofla2} and \eqref{eqestiofu2}. We point out that \eqref{eqestiofla2} and \eqref{eqestiofu2} were already obtained in  \cite{WW}. The approach of \cite{WW} relies on the clever introduction of some associated unconstrained problems. We provide here alternative proofs.

To obtain solutions with a prescribed $L^2$ norm, the strategy of first studying the problem on bounded domains and then passing to the limit has not yet been used much. In this direction, we know only \cite{BaQiZo} and the implementation of this approach for the study of the solutions of \eqref{eqequation} is then one of the major interests of our work. Solving first the problem on balls $B_{R_n}$ gives us two advantages. Firstly, when passing to the limit $R_n \to \infty$, since the sequences $\{u_{1,n}\}, \{u_{2,n}\}$  consist of exact solutions, the Pohozaev identity tells us that these sequences are bounded.  Secondly, since the elements of these sequences have symmetry, this provides us with the compactness property that they strongly converge in $L^{p+1}(\R^3)$.

In the last part of this work we briefly discuss the orbital stability of the solutions obtained in Theorem 1.2. It had already been established in \cite{BBJV} (see also \cite{SY}), that the nonempty set $\{u \in \G_\mu : E(u) = c_\mu\}$ is stable under the flow associated with \eqref{timenls} for $0 < \mu < \mu_1$.
In Remark \ref{stability} we explain how to prove that the standing wave corresponding to a solution at the mountain pass level $m_\mu$ is orbitally unstable for sufficiently small $\mu > 0$. We conjecture that this perturbative instability result should remain true for all $0 < \mu < \mu_1$. A new type of approach will probably have to be developed to establish it.

The paper is organized as follows. In Section \ref{Preliminaries}  we present the key tools, Lemma \ref{lemstrongconvergence}, dealing with the strong convergence in $L^{p+1}(\R^3),$ and the Liouville result, Proposition \ref{propnoposisolu}. We also explicit the geometry of $E$ on $S_{\mu}.$ In Section \ref{secBR} we focus on the problem on $B_R$ with sufficiently large $R$ and provide the proof of Theorem \ref{thmmultionBR} and Theorem \ref{thmasymonBR}. Section \ref{secR3} is devoted to the proof of Theorem \ref{thmmulti} and Theorem \ref{thmasym}. In Subsection \ref{subseccompactness}, we prove the compactness result Theorem \ref{thmcompactness}. Then, we complete the proof of Theorem \ref{thmmulti} and Theorem \ref{thmasym} in Subsection \ref{subsecproof1} and in Subsection \ref{subsecproof2} respectively. Remark \ref{stability} is finally given.

\section{Preliminaries} \label{Preliminaries} 

The results presented in this section will prepare the proof of Theorem \ref{thmmulti}. The following proposition is a particular case of \cite[Théorème III.2]{Li}. It shows how symmetry properties translate into compactness properties.
\begin{proposition} \label{propsymtocom}
	Let $2 \leq d < N$. Assume that $\sup \|u_n\|_{H^1(\R^N)} < \infty$ and that each $u_n$ is radially symmetric w.r.t. $(x_1,\cdots,x_d)$ and decreasing w.r.t. $|x_i|$, $i = d+1,\cdots,N$, then
	\begin{align*}
		u_n \rightharpoonup u  \quad \text{in } H^1(\R^N) \, \Longrightarrow \, u_n \to u  \quad \text{in } L^q(\R^N) \text{ for any } 2 < q < 2_N^* := \frac{2N}{N-2}.
	\end{align*}
\end{proposition}

Proposition \ref{propsymtocom} corresponds to the case $n=0$, $m=1$, $N_1=2$ and $p=1$ in \cite[Théorème III.2]{Li}. The proof of \cite[Theorem III.2]{Li}, which was not given, was supposed to be based on certain estimates of decay. We give here a proof that does not use such estimates. 

\begin{proof}[Proof of Proposition \ref{propsymtocom}]
	Let $\tilde u_n = u_n - u \rightharpoonup 0$ in $H^1(\R^N)$. By contradiction, we assume that, up to a subsequence 
	\begin{align*}
	\lim_{n \to \infty}\int_{\R^3}|\tilde u_n|^qdx > \epsilon \quad \text{for some } q \in (2,2_N^*) \text{ and } \epsilon > 0.
	\end{align*}
	Then, it is well known that there exists a sequence $\{y_n\} \subset \R^N$ such that 
	\begin{align*}
	    \tilde u_n(\cdot + y_n) \rightharpoonup \tilde u \neq 0, \quad \text{in } H^1(\R^N).
	\end{align*}
    Denote $y_n = (\bar y_n,\hat y_n)$ where $\bar y_n = (y_{n,1},\cdots,y_{n,d})$ and  $\hat y_n = (y_{n,d+1},\cdots,y_{n,N})$. We claim that $\{\bar y_n\}$ is a bounded sequence in $\R^d$. By negation, we assume that $|\bar y_n| \to \infty$. Take a bounded subset $\Omega$ of $\R^d$ such that $\int_{\Omega \times \R^{N-d}}|\tilde u|^qdx > 2\delta$ for some $\delta > 0$. Then we have
    \begin{align*}
    	\liminf_{n \to \infty}\int_{(\bar y_n+\Omega) \times \R^{N-d}}|\tilde u_n|^qdx = \liminf_{n \to \infty}\int_{\Omega \times \R^{N-d}}|\tilde u_n(\cdot+y_n)|^qdx \geq \int_{\Omega \times \R^{N-d}}|\tilde u|^qdx > 2\delta.
    \end{align*}
    So there exists $M < \infty$ such that $\int_{(\bar y_n+\Omega) \times \R^{N-d}}|\tilde u_n|^qdx > \delta$ for all $n > M$. Since $\Omega$ is bounded in $\R^d$ and $|\bar y_n| \to \infty$, for any $k \in \mathbb{N}$, we can take $M' \geq M$ (depending on $k$) and $\sigma_1,\cdots,\sigma_k$, where $\sigma_i : \R^d \to \R^d$ is a rotation map, $i = 1,\cdots,k,$ such that $\Omega_i \cap \Omega_j = \emptyset$ whenever $i \neq j$ where $\Omega_s = \sigma_s(\bar y_n+\Omega)$, $s=1,\cdots,k.$ It is clear that $\int_{\Omega_i \times \R^{N-d}}|\tilde u_n|^qdx > \delta$ when $n > M'$, $i = 1,\cdots,k,$ and thus $\int_{\R^N}|\tilde u_n|^qdx > k\delta$ when $n > M'$. Taking $k$ large enough such that $k\delta > \sup_n\int_{\R^N}|\tilde u_n|^qdx$ we obtain a contradiction. Thus $\{\bar y_n\}$ is a bounded sequence in $\R^d$. Since $\tilde u_n$ to $0$ weakly in $H^1(\R^N)$, the sequence $\{y_n\}$ is not bounded in $\R^N$ and so $\{\hat y_n\}$ is not bounded in $\R^{N-d}$. Also without loss of generality, we assume that $\bar y_n = 0$ for each $n$ and that $\{y_{n,d+1}\}$ is unbounded.
    Up to a subsequence, we further assume $\tilde u_n(x+y_n) \to \tilde u$ a.e. $x \in \R^N$. Without loss of generality, we assume that $\tilde u_n(y_n) \to \tilde u(0) \neq 0$. From the fact that $u_n$ is decreasing w.r.t. $|x_i|$, $i = d+1,\cdots,N$, we deduce that $u_n \geq 0$ and thus $u \geq 0$. As proved previously, the sequence $\{y_n\}$ is unbounded and so $u(y_n) = o_n(1)$. Then, noticing $u_n(y_n) = \tilde u_n(y_n) + u(y_n) = \tilde u_n(y_n) + o_n(1) \geq 0$, we obtain that $\tilde u(0) > 0$. Take $x^* = (0,\cdots,0,x_{d+1}^*,0,\cdots,0)$ such that $0 \leq u(x^*) < 1/3\tilde u(0)$. So $u_n(x^*) = \tilde u_n(x^*) + u(x^*) < 1/3\tilde u(0) + o_n(1)$. For large $n$ such that $|y_{n,d+1}| > |x_{d+1}^*|$, by the monotonicity of $u_n$ w.r.t. $|x_i|$, $i = d+1,\cdots,N,$ we obtain that $u_n(y_n) \leq u_n(x^*) < 1/3\tilde u(0) + o_n(1)$. However, this contradicts that $u_n(y_n) = \tilde u_n(y_n) + u(y_n) \to \tilde u(0) > 0$ completing the proof.
\end{proof}


As a direct consequence of Proposition \ref{propsymtocom} we have the following result.
\begin{lemma} \label{lemstrongconvergence}
Assume that $\sup \|u_n\|_{H} < \infty$ and that each $u_n$ is radially symmetric and decreasing w.r.t. $(x_1,x_2)$ and w.r.t. $x_3$ respectively, then
    \begin{align*}
		u_n \rightharpoonup u  \quad \text{in } H \, \Longrightarrow \, u_n \to u  \quad \text{in } L^{p+1}(\R^3).
		\end{align*}
\end{lemma}

\begin{remark}\label{usesymmetry}
Using, in particular, some of the results of \cite{Li}, many works have exploited the possibility of working in certain subspaces of functions with a certain symmetry in order to benefit from additional compactness properties.  In this direction we mention \cite{BaWi,BeLi,AvPoSc,JeLu,Me}. However, it seems to us that this is the first time that the compact inclusion stated in Lemma \ref{lemstrongconvergence} has been used, and moreover through a process of approximation.
\end{remark}
   
Next, we recall a result comparing the quantities associated with two spectral problems defined respectively in $3d$ and in $2d$.

\begin{lemma}[Lemma 2.1 in \cite{BBJV}] \label{lemspectrum}
	We have the following equality
	$$\La_0 = \la_0,$$
	where
	\begin{align*}
		\La_0 = \inf_{\int_{\R^3}|w|^2dx = 1}\int_{\R^3}\left(|\nabla w|^2 +(x_1^2+x_2^2)|w|^2 \right)dx, 
	\end{align*}
    and
    \begin{align*}
    	\la_0 = \inf_{\int_{\R^2}|v|^2dx_1dx_2 = 1}\int_{\R^2}\left(|\nabla_{x_1,x_2} v|^2 +(x_1^2+x_2^2)|v|^2 \right)dx_1dx_2, 
    \end{align*}
    are the infimum of spectrum defined in 3d and in 2d respectively.
\end{lemma}

\begin{remark} \label{rmkef2d}
	In the 2-dimensional case, it is well known that $\la_0 = 2$, it is simple and a corresponding minimizer is given by the gaussian function $e^{-(x_1^2+x_2^2)}$, see \cite[Section 2]{BBJV}.
\end{remark}

The following Liouville type result will be crucially used.

\begin{proposition} \label{propnoposisolu}
	If a non-negative continuous function $u \in H$ is radially symmetric and decreasing w.r.t. $(x_1,x_2)$ and w.r.t. $x_3$ respectively, and, weakly solves
	\begin{align} \label{eqweak}
		-\Delta u + Vu - \La_0 u = u^p \quad \text{in } \R^3,
	\end{align}
    then $u \equiv 0$. 
\end{proposition}


\begin{proof}
	Arguing by contradiction, we assume that $u \not\equiv 0$. It is easy to see that $u > 0$. By Lemma \ref{lemspectrum} and Remark \ref{rmkef2d}, we get
	\begin{align*}
		\left( -\Delta_{x_1,x_2} + (x_1^2+x_2^2)\right)e^{-(x_1^2+x_2^2)} = \La_0 e^{-(x_1^2+x_2^2)} \quad \text{in } \R^2.
	\end{align*}
    Let us define
    \begin{align*}
    	\varphi_R(x_3) :=
    	\begin{cases}
    		\sqrt{\frac{2R}{\pi}}\cos(\frac{\pi}{2R}x_3), & |x_3| < R, \\
    		0, & |x_3| \geq R.
    	\end{cases}
    \end{align*}
    It is clear that
    \begin{align*}
    	-\varphi_R''(x_3) = \frac{\pi^2}{4R^2}\varphi_R(x_3) \quad \text{in } \bigl\{|x_3| < R\bigr\}.
    \end{align*}
    Note that $e^{-(x_1^2+x_2^2)}\varphi_R(x_3) \in H$. Since $u$ weakly solves \eqref{eqweak}, we get
    \begin{align*}
    	& \int_{\R^3}\left( \nabla u \nabla e^{-(x_1^2+x_2^2)}\varphi_R(x_3) + Vue^{-(x_1^2+x_2^2)}\varphi_R(x_3) - \La_0 ue^{-(x_1^2+x_2^2)}\varphi_R(x_3)\right) dx \nonumber \\
    	= & \int_{\R^3}u^pe^{-(x_1^2+x_2^2)}\varphi_R(x_3)dx.
    \end{align*}
    We set
    \begin{align*}
    	& w(x_3) := \int_{\R^2}u(x_1,x_2,x_3)e^{-(x_1^2+x_2^2)}dx_1dx_2 \geq 0, \\
    	& v(x_3) := \int_{\R^2}u^p(x_1,x_2,x_3)e^{-(x_1^2+x_2^2)}dx_1dx_2 \geq 0.
    \end{align*}
     Since $u \in H$ is radially symmetric and decreasing w.r.t. $(x_1,x_2)$ and w.r.t. $x_3$ respectively, both $w$ and $v$ are symmetric and decreasing w.r.t. $x_3$. Direct calculations yield
     \begin{align*}
     	& \int_{\R^3}\left( \nabla u \nabla e^{-(x_1^2+x_2^2)}\varphi_R(x_3) + Vue^{-(x_1^2+x_2^2)}\varphi_R(x_3) - \La_0 ue^{-(x_1^2+x_2^2)}\varphi_R(x_3)\right) dx \nonumber \\
     	= & \frac{\pi^2}{4R^2}\int_{-R}^R w(x_3)\varphi_R(x_3)dx_3 - \sqrt{\frac{2\pi}{R}}w(R),
     \end{align*}
     and
     \begin{align*}
     	\int_{\R^3}u^pe^{-(x_1^2+x_2^2)}\varphi_R(x_3)dx = \int_{-R}^R v(x_3)\varphi_R(x_3)dx_3.
     \end{align*}
     Thus, we have
     \begin{align*}
     	\frac{\pi^2}{4R^2}\int_{-R}^R w(x_3)\varphi_R(x_3)dx_3 = \sqrt{\frac{2\pi}{R}}w(R) + \int_{-R}^R v(x_3)\varphi_R(x_3)dx_3, 
     \end{align*}
     that is, 
     \begin{align}\label{eqestiofw}
     	\left( \frac{\pi}{2R}\right) ^{\frac52}\int_{-\frac\pi2}^{\frac\pi2} w(\frac{2R}{\pi}x_3)\cos(x_3)dx_3 = \sqrt{\frac{2\pi}{R}}w(R) + \sqrt{\frac{\pi}{2R}}\int_{-\frac\pi2}^{\frac\pi2} v(\frac{2R}{\pi}x_3)\cos(x_3)dx_3. 
     \end{align}
     It is easy to check that $w(R) \geq 0$ and 
     $$
     \int_{-\frac\pi2}^{\frac\pi2} w(\frac{2R}{\pi}x_3)\cos(x_3)dx_3 \to 0 \quad \text{as } R \to \infty.
     $$
     From $u > 0$ we deduce that $v(0) > 0$. Moreover, $v$ is continuous at $x_3 = 0$ since $u$ is continuous, and for any $\epsilon > 0$ and for sufficiently large $R$ (depending on $\epsilon$), we have
     \begin{align} \label{eqestiofv}
     	\int_{-\frac\pi2}^{\frac\pi2} v(\frac{2R}{\pi}x_3)\cos(x_3)dx_3 \geq \int_{-R^{-(1+\epsilon)}}^{R^{-(1+\epsilon)}} v(\frac{2R}{\pi}x_3)\cos(x_3)dx_3 \geq v(0)R^{-(1+\epsilon)}.
     \end{align}
     Next, we fix $0 < \epsilon < 1$. Combining \eqref{eqestiofw} and \eqref{eqestiofv}, we conclude that
     \begin{align*}
     	\frac{4R^{1-\epsilon}}{\pi^2}
     	v(0) \leq \int_{-\frac\pi2}^{\frac\pi2} w(\frac{2R}{\pi}x_3)\cos(x_3)dx_3 \xrightarrow{R \to \infty} 0,
     \end{align*}
     which is self-contradictory and completes the proof.
\end{proof}

We end this section with two lemmas which explicit the geometry of the energy functional $E$ on $S_{\mu}$.

\begin{lemma} \label{lempropofg}
	Let $7/3 < p < 5$ and $0 < \mu < \mu_{1}$, where $\mu_1$ is defined in \eqref{eqmu1}. Then $\G_{\mu}$, defined in \eqref{eqG}, is nonempty and 
	$$0 < c_{\mu}< \frac12\La_0\mu < \nu_{\mu},$$
where $c_{\mu}$ and $\nu_{\mu}$ are defined in \eqref{min_boundary}.
\end{lemma}

\begin{proof}
	We set 
    $$
    w(x) = \Psi_0(x_1,x_2)\varphi(x_3), \quad \int_{\R}|\varphi|^2dx_3 = \mu,
    $$ 
    where $\Psi_0(x_1,x_2)$ is the unique normalized positive eigenvector of the quantum harmonic oscillator $-\Delta_{x_1,x_2} + x_1^2+x_2^2$ and $\varphi(x_3)$ will be chosen later. It is clear that
    \begin{align*}
    	\int_{\R^3}|\nabla w|^2dx = & \int_{\R^2}|\nabla_{x_1,x_2}\Psi_0(x_1,x_2)|^2dx_1dx_2\int_\R|\varphi(x_3)|^2dx_3 \\
    	& + \int_{\R^2}|\Psi_0(x_1,x_2)|^2dx_1dx_2\int_\R|\varphi'(x_3)|^2dx_3, \\
    	\int_{\R^3}|w|^{p+1}dx = & \int_{\R^2}|\Psi_0(x_1,x_2)|^{p+1}dx_1dx_2\int_\R|\varphi(x_3)|^{p+1}dx_3.
    \end{align*}
     Let $v(x_3) \in H^1(\R)$ with $\int_{\R}|v|^2dx_3 = \mu$ be non-negative and define $v_t := \sqrt{t}v(tx_3)$. Note that
		\begin{align*}
    	& \int_{\R}|v_t|^2dx_3 = \mu, \quad  \int_{\R}|(v_t)'|^2dx = t^2\int_{\R}|v'|^2dx, \quad \mbox{and} \quad \int_{\R}|v_t|^{p+1}dx = t^{\frac{p-1}{2}}\int_{\R}|v|^{p+1}dx.
    \end{align*}
    By taking $\varphi = v_t$ with $t > 0$ small enough, we have $w \in \G_\mu$ and thus $\G_\mu$ is not empty. Moreover,
    \begin{align*}
    	E(w) = \frac12\La_0\mu & + \frac12\int_{\R^2}|\Psi_0(x_1,x_2)|^2dx_1dx_2\int_\R|\varphi'(x_3)|^2dx_3 \\
    	& - \frac1{p+1}\int_{\R^2}|\Psi_0(x_1,x_2)|^{p+1}dx_1dx_2\int_\R|\varphi(x_3)|^{p+1}dx_3
    \end{align*}
    and thus $E(w) < 1/2\La_0\mu$ since $(p-1)/2 < 2$, which implies that $c_\mu < \frac{1}{2}\La_0\mu$. Now, for any $u \in \partial \G_{\mu}$, by the Gagliardo-Nirenberg inequality \eqref{eqgninequ} we have
    \begin{align*}
    	\int_{\R^3}|\nabla u|^2dx & = \frac{3p-3}{2p+2}\int_{\R^3}|u|^{p+1}dx \\
    	& \leq \mu^{\frac{5-p}4}C_{p+1}\frac{3p-3}{2p+2}\left(\int_{\R^3}|\nabla u|^2dx\right)^{\frac{3p-3}{4}},
    \end{align*}
    implying that
     \begin{align*}
    	\int_{\R^3}|\nabla u|^2dx \geq \left(C_{p+1}\frac{3p-3}{2p+2} \right) ^{-\frac4{3p-7}}\mu^{-\frac{5-p}{3p-7}}.
    \end{align*}
    Then for $0 < \mu < \mu_{1}$, we get using the definition of $\mu_1 >0$, see \eqref{eqmu1},
    \begin{align} \label{eqnumu}
    	\nu_{\mu} & \geq \frac{3p-7}{6p-6}\inf_{u \in \partial \G_{\mu}}\int_{\R^3}|\nabla u|^2dx \nonumber \\
    	& \geq \frac{3p-7}{6p-6}\left(C_{p+1}\frac{3p-3}{2p+2} \right) ^{-\frac4{3p-7}}\mu^{-\frac{5-p}{3p-7}} \nonumber \\
    	& > \frac12\La_0\mu.
    \end{align}
    Finally, it is easy to check that $c_\mu > 0$, by using the definition of $\G_\mu$.
\end{proof}

\begin{lemma} \label{lemmps}
	Let $7/3 < p < 5$ and $0 < \mu < \mu_{1}$, where $\mu_1$ is defined in \eqref {eqmu1}. There exists $0 \leq w_{0} \in \G_{\mu}$ and $0 \leq w_{1} \in S_{\mu} \backslash \G_{\mu}$, having  supports in $B_{R_0} \subset \R^3$, for some $R_0 >0$, such that
	\begin{align}\label{L1}
		m_{\mu} := \inf_{\ga\in \Ga}\sup_{t\in[0,1]}E(\ga(t)) \geq \nu_{\mu} > \max\bigl\{E(w_{0}), E(w_{1})\bigr\}.
	\end{align}
    In particular, $m_{\mu} > 1/2\La_0\mu$. Here $\G_{\mu}$ is defined in \eqref{eqG}, $\Ga$ in \eqref{eqGa} and $\nu_{\mu}$ in \eqref{min_boundary}.
\end{lemma}

\begin{proof}
	
	According to the proof of Lemma \ref{lempropofg}, there exists a non-negative $w \in \G_\mu$ such that $E(w) < 1/2\La_0\mu < \nu_\mu$ for $0 < \mu < \mu_1$. Without loss of generality, we further assume $supp \ w \subset B_{R_0}$ for some $R_0 > 0$. We choose $w_0 = w$. Now, for an arbitrary non-negative $v \in S_{\mu}$ having support in $B_{R_0},$ we define $v_t(x) := t^{3/2}v(tx)$. Clearly $v_t \in S_{\mu}$ and $supp \ v_t \subset B_{R_0}$  for all $t > 1$. By direct calculations
	\begin{align*}
		\int_{\R^3}|\nabla v_t|^2dx = t^2	\int_{\R^3}|\nabla v|^2dx, \quad \int_{\R^3}|v_t|^{p+1}dx = t^{\frac{3p-3}2}\int_{B_R}|v|^{p+1}dx.
	\end{align*}
   Also,
    \begin{align*}
		E_{R,\tau}(v_t) & \leq \frac12\int_{\R^3}(|\nabla v_t|^2 + V|v_t|^2)dx - \frac1{2p+2}\int_{\R^3}|v_t|^{p+1}dx \\
		& = \frac{t^2}2\int_{\R^3}|\nabla v|^2dx + \frac1{2t^2}\int_{\R^3}V|v_t|^2dx - \frac{t^{\frac{3p-3}2}}{2p+2}\int_{\R^3}|v_t|^{p+1}dx.
	\end{align*}
    Since $(3p-3)/2 > 2$, taking $w_{1} = v_t$ for sufficiently large $t$, we obtain $w_{1} \in S_{\mu} \backslash \G_{\mu}$, such that $E(w_{1}) < 1/2\La_0\mu < \nu_{\mu}$.
	Moreover, by continuity, for any $\ga \in \Ga$, there exists $t \in [0,1]$ such that $\ga(t) \in \partial \G_{\mu}$, and thus,
	\begin{align*}
		m_{\mu} = \inf_{\ga\in \Ga}\sup_{t\in[0,1]}E(\ga(t)) \geq \nu_{\mu} > \max\bigl\{E(w_{0}), E(w_{1})\bigr\}.
	\end{align*}
	Finally, by \eqref{eqnumu} we know that $\nu_\mu > 1/2\La_0\mu$ for $0 < \mu < \mu_1$, which completes the proof.
\end{proof}

\section{Problems on $B_R$} \label{secBR} 

Since the embedding $H^1_0(B_R) \hookrightarrow L^q(B_R)$ is compact for all $1 \leq q < 6$, any bounded Palais-Smale sequence of $E_R$ constrained on $S_{\mu,R}$  has a convergent subsequence in $S_{\mu,R}$. Such result is now standard, see, e.g., Proposition 3.1 in \cite{Dov}, but we provide here a proof for completeness.

\begin{lemma} \label{lemcompactonBR}
	Let $1 < p < 5$. Assume that $\{u_n\} \subset S_{\mu,R}$ is a bounded Palais-Smale sequence of $E_R$ constrained on $S_{\mu,R}$. Then there exists $u \in S_{\mu,R}$ such that, up to a subsequence, $u_n \to u$ strongly in $H^1_0(B_R)$.
\end{lemma}

\begin{proof}
	Since $\{u_n\}$ is bounded in $H^1_0(B_R)$, there exists $u \in H^1_0(B_R)$ such that $u_n \rightharpoonup u$ weakly in $H^1_0(B_R)$ as $n \to \infty$. By Sobolev compact embeddings, up to a subsequence, $u_n \to u$ strongly in $L^q(B_R)$ for all $1 \leq q < 6$, so that $u \in S_{\mu,R}$. 
	For any $w \in S_{\mu,R}$ and $\la \in R$, we define the linear functional $J_\la(w) : H^1_0(B_R) \to \R$ by
	$$
	J_\la(w)[v] := \int_{B_R}\left( \nabla w \nabla v + Vwv + \la wv - |w|^{p-1}wv \right) dx.
	$$
	Since $(E_R|_{S_\mu})'(u_n) \to 0$, there exists $\{\la_n\} \subset \R$ such that $J_{\la_n}(u_n) \to 0$ in $H^{-1}(B_R)$ as $n \to \infty$. Since $\{u_n\}$ is bounded in $H^1_0(B_R)$, it is clear that $\la_n$ is bounded. Note that $A(u)(u_n-u) \to 0$ by weak convergence of $u_n$ to $u$ where $A(u) : H^1_0(B_R) \to \R$ is defined by 
	$$
	A(u)[v] := \int_{B_R}\left(\nabla u \nabla v + Vuv \right)dx.
	$$
	Then we have
	\begin{align*}
		o(1) & = \left (J_{\la_n}(u_n)-A(u)\right) [u_n-u] \\
		& = \int_{B_R}\left( |\nabla (u_n-u)|^2 + V|u_n-u|^2\right) dx \\
		& = \int_{B_R}|\nabla (u_n-u)|^2dx,
	\end{align*}
    by the strong convergence in $L^2(B_R)$ and in $L^{p+1}(B_R)$ of $u_n$ to $u$ and by the boundedness of $\la_n$. This concludes the proof.
\end{proof}

\begin{lemma} \label{lempropofgmu}
	Let $7/3 < p < 5$ and $0 < \mu < \mu_{1}$. For $R_0 >0$ given in Lemma \ref{lemmps}, we have that $\G_{\mu,R}$ is nonempty and $0 < c_{\mu,R} < \nu_{\mu}$, for all $R > R_0$. In addition, any minimizing sequence of $E_R$ in $\G_{\mu,R}$ at the level $c_{\mu,R}$ is bounded in $H^1_0(B_R)$.
\end{lemma}

\begin{proof}
		Let $0 \leq w_0 \in \G_\mu$ be given by Lemma \ref{lemmps}. For all $R > R_0$, $w_0 \in \G_{\mu,R}$ and thus $\G_{\mu,R}$ is not empty. 
		From \eqref{L1}, we have
		\begin{align*} 
    	\nu_{\mu} > E(w_0) = E_R(w_0) \geq c_{\mu,R}.
    \end{align*}
    Clearly also $c_{\mu,R} > 0$.
    Finally, if $\{u_n\} \subset \G_{\mu,R}$ is a minimizing sequence of $E_R$ in $\G_{\mu,R}$, we have
    \begin{align*}
    	\int_{B_R}|\nabla u_n|^2dx < \frac{6p-6}{3p-7}E_R(u_n) \to \frac{6p-6}{3p-7}c_{\mu,R} \quad \text{as } n \to \infty,
    \end{align*}
    yielding the boundedness of $\{u_n\}$ in $H^1_0(B_R)$.
\end{proof}

Combining Lemma \ref{lemcompactonBR} and Lemma \ref{lempropofgmu}, we obtain the existence of a ground state immediately. To complete the proof of Theorem \ref{thmmultionBR}, we focus on the existence of a critical point at the level $m_{\mu,R}$. Thanks to Lemma \ref{lemcompactonBR}, the only difficulty is to get a bounded Palais-Smale sequence. This will be done using the monotonicity trick. Following the strategy laid down in \cite{Louis}, we introduce a family of functionals with a parameter $\tau \in [1/2,1]$:
\begin{align*} 
	E_{R,\tau}(u) = \frac12\int_{B_R}(|\nabla u|^2 + V|u|^2)dx - \frac\tau{p+1}\int_{B_R}|u|^{p+1}dx, \quad \frac73 < p <5.
\end{align*}

\begin{lemma} \label{lemmpsonBR}
	Let $7/3 < p < 5$ and $0 < \mu < \mu_{1}$. Let $R_0 >0$, $0 \leq w_0$ and $0 \leq w_1$ be given by Lemma \ref{lemmps}. Then, taking  $\varepsilon_0 = \varepsilon_0(\mu) > 0$ sufficiently small, we have that, for all $R >R_0$
	\begin{align}\label{L2}
		m_{\mu,R,\tau} := \inf_{\ga\in \Ga_R}\sup_{t\in[0,1]}E_{R,\tau}(\ga(t)) \geq \nu_{\mu} > \max\bigl\{E_{R,\tau}(w_{0}), E_{R,\tau}(w_{1})\bigr\}, \quad \forall \tau \in [1-\varepsilon_0,1],
	\end{align}
	where $\Ga_R = := \bigl\{\ga \in C([0,1],S_{\mu,R}): \ga(0) = w_{0}, \ga(1) = w_{1}\bigr\}.$
	\end{lemma}
	
\begin{proof}
Clearly $\Ga_R \neq \emptyset$ and $\Ga_R \subset \Ga$. Also $E_{\tau}(u)\geq E(u)$ for all $u \in S_{\mu}$ and $0\leq \tau \leq 1$. Thus the inequality $m_{\mu,R,\tau} \geq \nu_{\mu}$ follows from the left inequality in \eqref{L1}. Now, for $\tau \leq 1$ sufficiently close to $1$, the right inequality in \eqref{L1} implies the right one in \eqref{L2}.
\end{proof}


\begin{lemma} \label{lemappropro}
	Let $7/3 < p < 5$, $0 < \mu < \mu_{1}$ and $R_0 >0$ be given by Lemma \ref{lemmps}.
	There exists $\varepsilon_0 = \varepsilon_0(\mu) > 0$ such that for all $R > R_0$ and for almost every $\tau \in [1-\varepsilon_0,1]$, $E_{R,\tau}$ constrained to $S_{\mu,R}$ has a positive critical point at the level $m_{\mu,R,\tau}$.
\end{lemma}

\begin{proof}
	By adapting the monotonicity trick \cite{Louis} on $E_{R,\tau}$ or applying \cite[Theorem 1.5]{BCJN}, together with Lemma \ref{lemmpsonBR}, we deduce that for almost every $\tau \in [1-\varepsilon_0,1]$ there exists a bounded Palais-Smale sequence $\{u_n\}$ for the constrained functional $E_{R,\tau}|_{S_{\mu,R}}$ at level $m_{\mu,R,\tau}$. Since $u \in S_{\mu,R} \Rightarrow |u| \in S_{\mu,R}$, $w_{0} \geq 0$, $w_{1} \geq 0$, the map $u \mapsto |u|$ is continuous, and $E_{R,\tau}(u) = E_{R,\tau}(|u|)$, it is then possible to choose $\{u_n\}$ with the property that $u_n \geq 0$ on $B_R$. Using Lemma \ref{lemcompactonBR} and the strong maximum principle, we get the desired conclusion.
\end{proof}

Now we are prepared to prove Theorem \ref{thmmultionBR}.

\begin{proof}[Proof of Theorem \ref{thmmultionBR}]
	For $0 < \mu < \mu_{1}$, by Lemma \ref{lemcompactonBR} and Lemma \ref{lempropofgmu}, it is standard to prove the existence of a positive critical point $u_{1,R} \in \G_{\mu,R}$ of $E_R$ constrained on $S_{\mu,R}$ with $E_R(u_{1,R}) = c_{\mu,R}$. If $v \in S_{\mu,R}$ is a critical point of $E_R$ constrained on $S_{\mu,R}$, Pohozaev identity yields that $v \in \G_{\mu,R}$ and thus $E_R(v) \geq c_{\mu,R}$. This implies that $u_{1,R}$ is a positive ground state of $E_R$ constrained on $S_{\mu,R}$.
	
	Next, we show the existence of a critical point at the mountain pass level. Lemma \ref{lemappropro} yields the existence of a sequence $\{\tau_n\}$ with $\tau \to 1^-$ as $n \to \infty$ and a sequence $\{u_n\}$ such that $u_n$ is a positive critical point of $E_{R,\tau_n}$ constrained on $S_{\mu,R}$ at the level $m_{\mu,R,\tau_n}$. It is well known see, \cite[Lemma 2.3]{Louis} that $m_{\mu,R,\tau} \to m_{\mu,R,1}$ as $\tau \to 1^-$. By applying the Pohozaev identity \eqref{eqpohozaevonBR} on $u_n$, we get
	\begin{align*}
		\int_{B_R}|\nabla u_n|^2dx > \frac{(3p-3)\tau_n}{2p+2}\int_{B_R}|u_n|^{p+1}dx,
    \end{align*}
    and
    \begin{align*}
    	\frac{3p-7}{6p-6}\int_{B_R}|\nabla u_n|^2dx + \frac12\int_{B_R}V|u_n|^2dx < E_{R,\tau_n}(u_n) = m_{\mu,R,\tau_n} \xrightarrow{n \rightarrow \infty} m_{\mu,R,1},
    \end{align*}
    which implies the boundedness of $u_n$ in $H^1_0(B_R)$. Using Lemma \ref{lemcompactonBR} and the strong maximum principle, we get a positive critical point $u_{2,R}$ of $E_R$ constrained on $S_{\mu,R}$ at the level $m_{\mu,R}$. This complete the proof.
\end{proof}

Finally in this section, we provide the proof of Theorem \ref{thmasymonBR}.

\begin{proof}[Proof of Theorem \ref{thmasymonBR}]
	Let $u_{1,R}$ and $u_{2,R}$ be the positive critical points given by Theorem \ref{thmmultionBR}. For $i = 1,2$, it is standard that there exists the Lagrange multiplier $\la_{i,R} = \la(u_{i,R})$ such that
	\begin{align}\label{L3}
		-\Delta u_{i,R} + Vu_{i,R} +\la_{i,R} u_{i,R} = |u_{i,R}|^{p-1}u_{i,R} \quad \text{in} \ B_R,
	\end{align}
	and $\la_{i,R} > - \La_{R}$ since the above equation has no positive solution in $H^1_0(B_R)$ when $\la \leq -\La_R$. 
	The fact that $u_{i,R} \in C^2(\overline{B_R})$ follows by standard elliptic regularity and that $u_{i,R} >0$ from the strong maximum principle.  To establish the symmetry of our solutions we make use of \cite[Corollary 1]{GNN}. We write \eqref{L3} as
	\begin{align*}
		\Delta u_{i,R} + f((x_1,x_2,x_3), u_{i,R}) = 0  \quad \text{in} \ B_R,
	\end{align*}
	with $f((x_1,x_2,x_3),u) = u^p - (x_1^2 + x_2^3) u - \lambda u.$ Clearly $f$ is invariant with respect to any rotation that leaves invariant the $x_3$ axis. Then from 
	\cite[Corollary 1]{GNN} (note that $(x_1^2 + x_2^2)$ has the good sign), we deduce that the function $u_{i,R}$ is symmetric with respect to any plane containing the $x_3$-axis and decreasing in the normal direction to these planes. Thus it is radial and decreasing with respect to $(x_1,x_2)$. Also, since $f$ does not depend on $x_3$, \cite[Corollary 1]{GNN} implies that $u$ is symmetric (even) with respect to $x_3$ and decreasing along this variable. At this point the proof of the theorem is completed.
\end{proof}

\section{Problems on $\R^3$} \label{secR3}

This section is devoted to the proof of Theorem \ref{thmmulti} and Theorem \ref{thmasym}.

\subsection{A compactness result} \label{subseccompactness}

The main result in this subsection is the following compactness result.

\begin{theorem} \label{thmcompactness}
	Let $\{u_n\}$ be a sequence such that $u_n \in S_{\mu,R_n}$ is a positive critical point of $E_{R_n}$ constrained on $S_{\mu,R_n}$, $R_n \to \infty$, $E_{R_n}(u_n) \to \beta \neq 1/2\La_0\mu$, as $n \to \infty$. Assume in addition, that, for each $n $, $u_{R_n}(x_1,x_2,x_3)$ is radially symmetric and 
	decreasing w.r.t. $(x_1,x_2)$ and w.r.t. $x_3$ respectively. Then, up to a subsequence, $u_n \to u$ strongly in $H$ and $u \in S_\mu$ is a positive critical point of $E$ constrained on $S_\mu$ at the level $\beta$.
\end{theorem}

Before proving Theorem \ref{thmcompactness}, we need some preparations.

\begin{lemma} \label{lemlimitoffirsteig}
	As $R \to \infty$, we have $\La_R \to \La_0$.
\end{lemma}

\begin{proof}
	For any $0 < R < R' < \infty$, it is standard to show that $\La_R > \La_{R'} > \La_0$. In particular $\lim_{R \to \infty}\La_{R} := \La^*$ exists and $\La^* \geq \La_0$. Let us assume, by contradiction, that $\La^* > \La_0$.
For $\epsilon > 0$ small enough such that $(\La_0 + 2\epsilon)/(1-\epsilon) < \La^*$, we take $\varphi \in H$ with 
	$$
	\int_{\R^3}|\varphi|^2dx = 1 \text{ and } \int_{\R^3}\left( |\nabla \varphi|^2 + V|\varphi|^2dx\right) dx < \La_0 + \epsilon.
    $$
    When $r > 0$ is sufficiently large, there exists $\chi \in C_0^\infty$ such that $\chi \equiv 1$ in $B_r$, $\chi \equiv 0$ in $\R^3 \backslash B_{r+1}$, $|\nabla \chi| \leq C$, and 
    \begin{align*}
    	\int_{\R^3}|\varphi\chi|^2dx := \mu_0 > 1 - \epsilon, \quad\quad \int_{\R^3}\left( |\nabla (\varphi\chi)|^2 + V|\varphi\chi|^2\right) dx \leq \La_0 + 2\epsilon.
    \end{align*}
     At this point, we define
    $$
    \tilde \varphi := \sqrt{\frac{1}{\mu_0}}\varphi\chi \in H^1_0(B_{r+1}).
    $$
    Then, we have 
    $$
    \int_{\R^3}|\tilde \varphi|^2dx = 1 \text{ and } \int_{\R^3}\left( |\nabla \tilde \varphi|^2 + V|\tilde \varphi|^2dx\right) dx \leq \frac{1}{1-\epsilon}\left( \La_0 + 2\epsilon\right) < \La^*,
    $$
    indicating that $\La_{r+1} < \La^*$, which is a contradiction to $\La^* \leq \La_{r+1}$ and completes the proof.
\end{proof}

Now we are prepared to prove the main result in this subsection.

\begin{proof}[Proof of Theorem \ref{thmcompactness}]
	We divide the proof into three steps.
	
	\vskip0.1in
	{\bf Step 1:} $\{u_n\}$ is bounded in $H$.
	
	Since $u_n$ is a critical point of $E_{R_n}$ constrained on $S_{\mu,R_n}$, by the Pohozaev identity \eqref{eqpohozaevonBR} on $B_{R_n}$ we get
	\begin{align*}
		\frac{3p-7}{6p-6}\int_{\R^3}|\nabla u_n|^2dx + \frac12\int_{\R^3}V|u_n|^2dx \leq E(u_n) \to \beta \quad \text{as } n \to \infty,
	\end{align*}
	which yields the boundedness of $\{u_n\} \subset H$.
	
	\vskip0.1in
	{\bf Step 2:} There exists $\epsilon_0 > 0$ such that $\int_{R^3}|u_n|^{p+1}dx > \epsilon_0$ for all $n$.
	
	Since $u_n$ is a critical point of $E_{R_n}$ constrained on $S_{\mu,R_n}$, it is standard that there exists the Lagrange multiplier $\la_n = \la(u_n)$ such that
	\begin{align*}
		-\Delta u_n + Vu_n +\la_n u_n = |u_n|^{p-1}u_n \quad \text{in} \ B_{R_n}.
	\end{align*}
    The boundedness of $\{u_n\}$ in $H$ indicates that $\{\la_n\}$ is bounded. Moreover, it is clear that $\la_n > -\La_{R_n}$. Up to a subsequence, we assume that $\la_n \to \la^*$ for some $\la^*$. By Lemma \ref{lemlimitoffirsteig}, $\la^* \geq -\La_0$. By negation, we suppose $\int_{\R^3}|u_n|^{p+1}dx \to 0$ up to a subsequence. Then, it is easy to see that
    \begin{align*}
    	\int_{\R^3}\left( |\nabla u_n|^2 + V|u_n|^2\right) dx \to -\la^* \mu \leq \La_0\mu \quad \text{as } n \to \infty.
    \end{align*}
    Moreover,
    \begin{align*}
    	\int_{\R^3}\left( |\nabla u_n|^2 + V|u_n|^2\right) dx \geq \La_0\mu.
    \end{align*}
    Thus, $\la^* = -\La_0$ and then $\beta = 1/2 \La_0 \mu$. However, we have assumed that $\beta \neq 1/2 \La_0 \mu$ in Theorem \ref{thmcompactness}. 
    \vskip0.1in
	{\bf Step 3:} Conclusion.
	
	By Step 1, Step 2 and Lemma \ref{lemstrongconvergence}, up to a subsequence, $u_n \rightharpoonup u \neq 0$ weakly in $H$ and $u_n \to u $ strongly in $L^{p+1}(\R^3)$. Moreover, we can assume that 
	$u_n \to u$ a.e. in $\R^3$. For any $\varphi \in C_0^\infty(\R^3)$, there exists $N_0$, such that, for all $n > N_0$, $supp \ \varphi \subset B_{R_n}$, and thus
	\begin{align*}
		\int_{\R^3}\left( \nabla u_n \nabla \varphi + Vu_n\varphi + \la_n u_n \varphi\right) dx = \int_{\R^3}|u_n|^{p-1}u_n\varphi dx.
	\end{align*}
    By the weak convergence $u_n \rightharpoonup u$ and the convergence $\la_n \to  \la^*$, we get
    \begin{align*}
    	\int_{\R^3}\left( \nabla u \nabla \varphi + Vu\varphi + \la^* u \varphi\right) dx = \int_{\R^3}|u|^{p-1}u\varphi dx, \quad \forall \varphi \in C_0^\infty(\R^3).
    \end{align*}
	Namely, $u$ is a weak solution to
	\begin{align*}
		- \Delta u + Vu + \la^* u = |u|^{p-1}u.
	\end{align*}
    Since $u_n$ is non-negative, $u$ is also non-negative. By strong maximum principle, $u > 0$ in $\R^3$. Further, since $u_n$, is radially symmetric and decreasing w.r.t. $(x_1,x_2)$ and w.r.t. $x_3$ respectively, it is also the case of $u$.
	By the standard elliptic regularity theory, we deduce that $u$ is continuous. Then, recalling that 
	$\la^* \geq \- \La_0$ we get by Proposition \ref{propnoposisolu}, $\la^* > -\La_0$.
	
	It remains to show that $u_n \to u$ in $H$. It is clear that
    \begin{align*}
    	\int_{\R^3}\left( |\nabla u_n|^2 + V|u_n|^2 + \la_n|u_n|^2 \right) dx = \int_{\R^3}|u_n|^{p+1}dx
    \end{align*}
    and
    \begin{align*}
    	\int_{\R^3}\left( |\nabla u|^2 + V|u|^2 + \la^*|u|^2 \right) dx = \int_{\R^3}|u|^{p+1}dx.
    \end{align*}
	Since $\la_n \to \la^*$, we thus have
		 \begin{align*}
		\int_{\R^3}\left( |\nabla u_n|^2 + V|u_n|^2 + \la^*|u_n|^2 \right) dx - \int_{\R^3}|u_n|^{p+1}dx \to \int_{\R^3}\left( |\nabla u|^2 + V|u|^2 + \la^*|u|^2 \right) dx - \int_{\R^3}|u|^{p+1}dx =0.
		 \end{align*}
		Now, using that $u_n \to u$ in $L^{p+1}(\R^3)$, see Lemma \ref{lemstrongconvergence}, necessarily
		\begin{align*}
		\int_{\R^3}\left( |\nabla u_n|^2 + V|u_n|^2 + \la^*|u_n|^2 \right) dx  \to \int_{\R^3}\left( |\nabla u|^2 + V|u|^2 + \la^*|u|^2 \right) dx.
		 \end{align*}
As $\la^* > - \La_0$, this convergence implies that $u_n \to u$ strongly in $H$. The proof is completed.
\end{proof}

\subsection{Proof of Theorem \ref{thmmulti}} \label{subsecproof1}

\begin{lemma} \label{lemlimitofcandm}
	Let $7/3 < p < 5$, $0 < \mu < \mu_{1}$ and $R_0 < R < R' < \infty$, where $R_0>$ is given by Lemma \ref{lemmps}. Then
	\begin{itemize}
		\item[(i)] $c_{\mu,R} > c_{\mu,R'} > c_{\mu}$ and $c_{\mu,R} \to c_\mu$ as $R \to \infty$;
		\item[(ii)] $m_{\mu,R} \geq m_{\mu,R'} \geq m_{\mu}$ and $m_{\mu,R} \to m_\mu$ as $R \to \infty$.
	\end{itemize}
\end{lemma}

\begin{proof}
	For $R_0 < R < R' < \infty$, it is easy to see that $\G_{\mu,R} \subset \G_{\mu,R'} \subset \G_{\mu}$ since any function in $H^1_0(B_R)$ can be viewed as a function in $H^1_0(B_R')$ and also in $H$ by zero extending. 
	Hence, we have $c_{\mu,R} \geq c_{\mu,R'} \geq c_{\mu}$. Next, we prove $c_{\mu,R} > c_{\mu,R'}$. By negation, $c_{\mu,R} = c_{\mu,R'}$. Using Theorem \ref{thmmultionBR}, we obtain a minimizer $u_{1,R} \in \G_{\mu,R}$ such that $E_{R}(u_{1,R}) = c_{\mu,R}$. Thus, $u_{1,R} \in \G_{\mu,R'}$ and $E_{R'}(u_{1,R}) = E_{R}(u_{1,R}) = c_{\mu,R} = c_{\mu,R'}$. It is standard to prove that $u_{1,R} > 0$ in $B_{R'}$, which contradicts to the fact that $u \equiv 0$ in $B_{R'} \backslash B_R$. Therefore, we prove $c_{\mu,R} > c_{\mu,R'}$. Similarly, we have $c_{\mu,R'} > c_{\mu}$. It is clear $\liminf_{R \to \infty}c_{\mu,R} \geq c_{\mu}$. To complete the proof of (i), it suffices to prove $\limsup_{R \to \infty}c_{\mu,R} \leq c_{\mu}$. By negation, $\exists R_n \to \infty$, $\lim_{n \to \infty}c_{\mu,R_n} = c_{\mu}^* > c_{\mu}$. For $\epsilon > 0$ small enough such that $c_{\mu} + \epsilon < c_{\mu}^*$, we take $\varphi \in \G_{\mu}$ with $E(\varphi) \leq c_{\mu} + \epsilon$. Without loss of generality, similar to the proof of Lemma \ref{lemlimitoffirsteig}, we assume that the support of $\varphi$ is compact and $supp \ \varphi \subset B_r$ for sufficiently large $r > 0$. Then $\varphi \in \G_{\mu,r}$ and
	\begin{align*}
		c_{\mu,r} \leq E_r(\varphi) = E(\varphi) \leq c_\mu + \varepsilon < c_\mu^*,
	\end{align*}
    which contradicts to $c_{\mu,r} > c_\mu^*$. 
    
    In the rest of this proof, we will show (ii). Clearly $\Ga_R \subset \Ga_{R'} \subset \Ga$ for $R_0 < R < R' < \infty$, and so $m_{\mu,R} \geq m_{\mu,R'} \geq m_{\mu}$. It is clear $\liminf_{R \to \infty}m_{\mu,R} \geq m_{\mu}$. To complete the proof of (ii), it suffices to prove $\limsup_{R \to \infty}m_{\mu,R} \leq m_{\mu}$. Note that $[0,1]$ is compact. Therefore reasoning as in the proof of (i) we can complete the proof.
\end{proof}

\begin{proof}[Proof of Theorem \ref{thmmulti}]
	Theorem \ref{thmmultionBR} provides a sequence $R_n \to \infty$, together with sequences $\{u_{1,n}\}, \{u_{2,n}\}$ such  that $u_{i,n} \in S_{\mu,R_n}$ is a positive critical point of $E_{R_n}$ constrained on $S_{\mu,R_n}$, $i = 1,2$, and $E_{R_n} (u_{1,n}) = c_{\mu,R_n}$, $E_{R_n} (u_{2,n}) = m_{\mu,R_n}$. In view of Lemma \ref{lempropofg}, Lemma \ref{lemmps} and Lemma \ref{lemlimitofcandm}, as $n \to \infty$, we have $E_{R_n} (u_{1,n}) \to c_\mu < 1/2\La_0\mu$ and $E_{R_n} (u_{2,n}) \to m_\mu > 1/2\La_0\mu$. At this point, by applying Theorem \ref{thmcompactness}, we get two positive critical points $u_1,u_2$ of $E$ constrained on $S_\mu$ with $E(u_1) = c_\mu$ and $E(u_2) = m_\mu$. Similar to the proof of Theorem \ref{thmmultionBR}, the fact $E(u_1) = c_\mu$ implies that $u_1$ is a ground state of $E$ constrained on $S_\mu$.
\end{proof}

\subsection{Proof of Theorem \ref{thmasym}} \label{subsecproof2}

\begin{proof}[Proof of Theorem \ref{thmasym}]

Recall in the proof of Theorem \ref{thmmulti} that $u_{1,n} \to u_1, u_{2,n} \to u_2$ strongly in $H$. By Theorem \ref{thmasymonBR}, $u_{1,n}$ and $u_{2,n}$, also $u_1$ and $u_2$, are radially symmetric and decreasing w.r.t. $(x_1,x_2)$ and w.r.t. $x_3$. For $i = 1,2$, it is standard that there exists Lagrange multiplier $\la_i = \la(u_i)$ such that
\begin{align*}
	-\Delta u_i + Vu_i +\la_i u_i = |u_i|^{p-1}u_i \quad \text{in} \ \mathbb{R}^{3}.
\end{align*}
For $i = 1$, on the one hand, multiplying by $u_1$ and integrating by parts we get
\begin{align*}
	-\la_1 = \mu^{-1}\int_{\R^3}\left(|\nabla u_1|^2 + V|u_1|^2 - |u_1|^{p+1}\right)dx. 
\end{align*}
Using Pohozaev identity \eqref{eqpohozaev}, we obtain
\begin{align*}
	\int_{\R^3}\left(|\nabla u_1|^2 + V|u_1|^2\right)dx < \frac{6p-6}{3p-7}E(u_1) = \frac{6p-6}{3p-7}c_{\mu} < \frac{3p-3}{3p-7}\La_0\mu.
\end{align*}
Then, by the Gagliardo-Nirenberg inequality \eqref{eqgninequ}, we have
\begin{align*}
	-\la_1 & \geq \mu^{-1}\left(\int_{\R^3}\left(|\nabla u_1|^2 + V|u_1|^2\right)dx - C_{p+1}\mu^{\frac{5-p}{4}}\left( \int_{\R^3}\left(|\nabla u_1|^2 + V|u_1|^2\right)dx\right) ^{\frac{3p-3}{4}} \right) \\
	& = \mu^{-1}\int_{\R^3}\left(|\nabla u_1|^2 + V|u_1|^2\right)dx \left(1- C_{p+1}\mu^{\frac{5-p}{4}}\left( \int_{\R^3}\left(|\nabla u_1|^2 + V|u_1|^2\right)dx\right) ^{\frac{3p-7}{4}} \right) \\
	& > \mu^{-1}\int_{\R^3}\left(|\nabla u_1|^2 + V|u_1|^2\right)dx \left(1- C_{p+1}\mu^{\frac{5-p}{4}}\left( \frac{3p-3}{3p-7}\La_0\mu\right) ^{\frac{3p-7}{4}} \right),
\end{align*}
and when $0 < \mu < C_{p+1}^{-\frac2{p-1}}\left( \frac{3p-3}{3p-7}\La_0\right)^{-\frac{3p-7}{2p-2}}$, we conclude that
\begin{align*}
	\la_1 < -\La_0\left(1- C_{p+1}\mu^{\frac{5-p}{4}}\left( \frac{3p-3}{3p-7}\La_0\mu\right) ^{\frac{3p-7}{4}} \right).
\end{align*}
On the other hand, 
\begin{align*}
	-\la_1 = \mu^{-1}\int_{\R^3}\left(|\nabla u_1|^2 + V|u_1|^2 - |u_1|^{p+1}\right)dx < 2\mu^{-1}E(u_1) = 2\mu^{-1}c_\mu < \La_0,
\end{align*}
which proves \eqref{eqestiofla1}. Repeating the steps in \cite[Subsection 4.4]{BBJV} we obtain \eqref{eqestiofu1}.

Next, in order to show \eqref{eqestiofla2} and \eqref{eqestiofu2}, we use some scaling arguments. The bijection 
$$
T_\mu : u(x) \mapsto \mu^{\frac{2}{3p-7}}u(\mu^{\frac{p-1}{3p-7}}x)
$$
continuously maps any function in $S_\mu$ to $S_1$. Set
\begin{align*}
	E_\mu(u) := \frac12\int_{\R^3}|\nabla u|^2dx + \frac{1}{2}\mu^\frac{4p-4}{3p-7}\int_{\R^3}V|u|^2dx - \frac1{p+1}\int_{\R^3}|u|^{p+1}dx.
\end{align*}
It is clear that
\begin{align*}
	E_\mu(T_\mu u) = \mu^{\frac{5-p}{3p-7}}E(u),
\end{align*}
and so $\widehat u_2 := T_\mu u_2$ is a critical point of $E_\mu$ constrained on $S_1$. Let us show that $\widehat u_2$ lies at a mountain-pass level. It is easy to see that $\G$ is invariant with respect to the map $T_\mu$ and its inverse, i.e., 
$$
u \in \G \Leftrightarrow T_\mu u \in \G.
$$
Hence,
$$
u \in \G_\mu \Leftrightarrow T_\mu u \in \G_1.
$$ 
Since $\Ga$, defined by \eqref{eqGa}, depends on $\mu$, we rename it by $\Ga_\mu$, expressed as
\begin{align*}
	\Gamma_\mu := \bigl\{\gamma \in C([0,1],S_\mu): \ga(0) = w_{\mu,0}, \ga(1) = w_{\mu,1} \bigr\}.
\end{align*}
We also introduce
\begin{align*}
	& m_{1,\mu} := \inf_{\ga\in\Ga_1}\sup_{t \in [0,1]}E_\mu(\ga(t)), \\
	& \nu_{1,\mu} := \inf_{u \in \partial \G_1}E_\mu(u) = \mu^{\frac{5-p}{3p-7}}\nu_\mu.
\end{align*}
Similar to \eqref{eqnumu} we have $\nu_{1,\mu} \geq \frac{3p-7}{6p-6}\left(C_{p+1}\frac{3p-3}{2p+2} \right) ^{-\frac4{3p-7}}$ for any $0 < \mu < \mu_1$. 
Take $\bar \mu \in (0,\mu_1)$ and recall that $w_{\bar \mu,0}$ and $w_{\bar \mu,1}$ are $w_0$ and $w_1$ given in Lemma \ref{lemmps} with $\mu = \bar \mu$ and 
\begin{align*}
	E(w_{\bar \mu,i}) < 1/2\La_0\bar\mu < \frac{3p-7}{6p-6}\left(C_{p+1}\frac{3p-3}{2p+2} \right) ^{-\frac4{3p-7}}\bar\mu^{-\frac{5-p}{3p-7}}, \quad i =0,1.
\end{align*}
Then, for $0 < \mu < \bar\mu$, we have
\begin{align*}
	E_\mu(T_{\bar\mu}w_{\bar \mu,i}) < E_{\bar \mu}(T_{\bar\mu}w_{\bar \mu,i}) = \bar\mu^{\frac{5-p}{3p-7}}E(w_{\bar \mu,i}) < \frac{3p-7}{6p-6}\left(C_{p+1}\frac{3p-3}{2p+2} \right) ^{-\frac4{3p-7}} \leq \nu_{1, \mu}, \quad i =0,1.
\end{align*}
By direct computation we get,
\begin{align*}
	E_\mu(T_\mu u) < \nu_{1,\mu} \Leftrightarrow E(u) < \nu_\mu.
\end{align*}
Thus $E(T_\mu^{-1} T_{\bar\mu}w_{\bar \mu,i}) < \nu_\mu$, $i=0,1.$ As mentioned previously, the set $\G$ is invariant with respect to the map $T_\mu$ and its inverse, and so $T_\mu^{-1} T_{\bar\mu}w_{\bar \mu,0} \in \G_\mu$ and $T_\mu^{-1} T_{\bar\mu}w_{\bar \mu,1} \in S_\mu \backslash \G_\mu$. Therefore, for $0 < \mu < \bar\mu$, we can take $w_{\mu,i} = T_\mu^{-1}T_{\bar\mu}w_{\bar \mu,i}$, $i=0,1.$ This indicates
$$
\ga \in \Ga_\mu \Leftrightarrow T_\mu \circ \ga \in \Ga_1,
$$
and thus $m_{1,\mu} = \mu^{\frac{5-p}{3p-7}}m_\mu$ when $0 < \mu < \bar\mu$, which shows that $\widehat u_2 = T_\mu u_2$ is a critical point of $E_\mu$ constrained on $S_1$ at the level $m_{1,\mu}$. It is standard to prove $\lim_{\mu \to 0}m_{1,\mu} \to m_{1,0}$ (similar to the proof of Lemma \ref{lemlimitofcandm}). For any sequence $\mu_n \to 0$, we set $u_n := T_{\mu_n} u_2$ and study its strong convergence in $H^1(\R^3)$ up to subsequences. There exists Lagrange multiplier $\omega_n \in \R$ such that
$$
-\Delta u_n + \mu_n^{\frac{4p-4}{3p-7}}Vu_n + \omega_n u_n = u_n^p, \quad u_n > 0, \quad \text{in } \R^3.
$$
Note that $\omega_n = \mu_n^{\frac{2p-2}{3p-7}}\la_2$. We claim that $\la_2$, depending on $\mu$, tends to infinity as $\mu \to 0$, which will be proved in Lemma \ref{lemla2tendstoinfinity} below. Particularly, $\la_2 > 0$ for sufficiently small $\mu > 0$. Hence, $\omega_n > 0$ for $n$ large enough.
From
$$
E_{\mu_n}(u_n) = m_{1,\mu_n} \to m_{1,0} \quad \text{as } n \to \infty,
$$
and $\int_{\R^3}|u_n|^2dx = 1$, we deduce that $u_n$ is bounded in $H^1(\R^3)$. Up to a subsequence, we assume $\omega_n \to \omega^* \geq 0$ and $u_n \rightharpoonup u$ weakly in $H^1(\R^3)$. At this point, let us prove there exists $\epsilon_0 > 0$ such that
\begin{align*}
	\int_{\R^3}|u_n|^{p+1}dx > \epsilon_0, \quad \forall n.
\end{align*}
By negation, we assume $\int_{\R^3}|u_n|^{p+1}dx \to 0$ passing a subsequence if necessary. Noticing that $\omega_n > 0$, we obtain $E_{\mu_n}(u_n) \to 0$, which is impossible since $E_{\mu_n}(u_n) = m_{1,\mu_n} \to m_{1,0} > 0$. Then, since the $u_{n}$ are radially symmetric and decreasing w.r.t. $(x_1,x_2)$ and w.r.t. $x_3$ we know by Lemma \ref{lemstrongconvergence} that $u \neq 0.$ It is easy to check $u \in H^1(\R^3)$ weakly solves
$$
-\Delta u + \omega^* u = u^p \quad \text{in } \R^3,
$$
and then $u > 0$ in $\R^3$. This in turn indicates that $\omega^* > 0$. Since $V(x) \geq 0$ and $\mu_n > 0$, for $n$ large enough we get
$$
-\Delta u_n + \frac12\omega^* u_n \leq u_n^p \quad \text{in } \R^3.
$$
Since $\{u_n\}$ is bounded in $H^1(\R^3)$, it follows from the classical elliptic estimates that $u_n$ exponentially decays uniformly with respect to $n$. Thus $\int_{\R^3}|Vu_n|^2dx$ is bounded and $\mu_n^{4p-4} \int_{\R^3}|Vu_n|^2dx$ tends to $0$. Then, it is standard to prove the strong convergence of $u_n$ to $u$ in $H^1(\R^3)$. Obviously, $\int_{\R^3}|u|^2dx = 1$, thus $\omega^* = \omega_1$ where $\omega_1$ is given by Remark \ref{rmkomega1} and $u = u_{\omega_1}$, the unique positive solution of
$$
-\Delta u + \omega_1 u = u^p \quad \text{in } \R^3, \quad\quad u(x) \to 0 \quad \text{as } |x| \to \infty.
$$
This completes the proof of \eqref{eqestiofla2} and \eqref{eqestiofu2}. 
\end{proof}

\begin{lemma} \label{lemla2tendstoinfinity}
	Let $\la_2$, depending on $\mu$, be given by Theorem \ref{thmasym} as the Lagrange multiplier of $u_2$. Then $\la_2 \to \infty$ as $\mu \to 0$.
\end{lemma}

\begin{proof}
	It is clear that $\la_2 > -\La_0$. Since $\la_2$ and $u_2$ depend on $\mu$, we rename them by $\la_{2,\mu}$ and $u_{2,\mu}$. Arguing by contradiction, we suppose there exists $\mu_n \to 0$ such that $\la_{2,\mu_n}$ is bounded. Using classical blow up analysis, see for example \cite{GiSp}, we deduce that $(u_{2,\mu_n})$ is uniformly bounded in $L^\infty(\R^3)$. Then, since $\int_{\R^3}|u_{2,\mu_n}|^2dx = \mu_n \to 0$, we get $\int_{\R^3}|u_{2,\mu_n}|^{p+1}dx \to 0$ as $n \to \infty$. This in turn indicates that $E(u_{2,\mu_n}) \to 0$ as $n \to \infty$, in a contradiction with
	$$
	E(u_{2,\mu}) = m_{\mu} \geq \nu_{\mu} \to \infty \quad \text{as } \mu \to 0,
	$$
	where $\nu_{\mu} \to \infty$ is deduced from \eqref{eqnumu}. This completes the proof.
\end{proof}

Finally we discuss the orbital stability of the standing wave corresponding to a solution at the mountain pass level $m_\mu$. 

\begin{remark}\label{stability}
It can be proved that the standing wave corresponding to a solution at the mountain pass level $m_\mu$ is orbitally unstable for $\mu > 0$ sufficiently small. As, starting from Theorems \ref{thmmulti} and \ref{thmasym}, the argument is relatively standard, we will just give the broad outlines. Let us first observe that, by \eqref{eqestiofla2}, $\la_2 \to \infty$ as $\mu \to 0$. Next we use the scaling $u_2 \mapsto \la_2^{-1/(p-1)}u_2(\la_2^{-1/2}x)$ and study the limit process as the parameter $\la_2$ to $\infty$. Noting \eqref{eqestiofu2}, it is readily seen that $\la_2^{-1/(p-1)}u_2(\la_2^{-1/2}x)$ converges to $u^*$ strongly in $H^1(\R^3)$ where $u^* \in H^1(\R^3)$ is the unique positive solution of the following equation
	\begin{align*}
		-\Delta u + u = |u|^{p-1}u, \quad u \in H^1(\R^3).
    \end{align*} 
    It is well known that $u^*$ is radially symmetric and nondegenerate. Now we rename $\la_2$ by $\la_{2,\mu}$ and the nondegeneracy of $u^*$ enables us to prove that $u_2$ is nondegenerate for small $\mu$ and show the existence of $\bar \mu \in (0,\mu_1]$ such that $\mu^* = \mu^{**}$ if $0 < \mu^*, \mu^{**} < \bar\mu$ and $\la_{2,\mu^*} = \la_{2,\mu^{**}}$. Thus for sufficient large $\la_2$ we can define a map $\la_2 \mapsto \int_{\R^3}|u_2|^2dx$, denoted by $\mu(\la_2)$. Clearly $\mu(\la_2)$ is monotonically decreasing. Using the implicit function theorem we obtain that $\mu(\la_2)$ is of class $C^1$ and so $\mu'(\la_2) \leq 0$. Further, borrowing the ideas of \cite[Proposition 11 (ii)]{FSK} and by increasing $\la_2$ if necessary, one can prove that $\mu'(\la_2) < 0$. Then, using the general results \cite[Theorem 2 and Theorem 3]{GSS} we conclude that the standing wave corresponding to a solution at the mountain pass level $m_\mu$ is orbitally unstable for sufficiently small $\mu > 0$.
\end{remark}

\end{document}